\documentclass[12pt,english]{article}
\usepackage[T1]{fontenc}
\usepackage[latin9]{luainputenc}
\usepackage{amsmath}
\usepackage{enumerate}
\usepackage{amsthm}
\usepackage{amssymb}
\usepackage[paper=a4paper, top=2.5cm, bottom=2.5cm, left=2.5cm, right=2.5cm]{geometry}

\makeatletter
\theoremstyle{plain}
\newtheorem{thm}{\protect\theoremname}
  \theoremstyle{plain}
  \newtheorem{cor}[thm]{\protect\corollaryname}
  \theoremstyle{plain}
  \newtheorem{lem}[thm]{\protect\lemmaname}
  \theoremstyle{plain}
  \newtheorem{prop}[thm]{\protect\propositionname}

\usepackage{babel}

\makeatother

\usepackage{babel}
  \providecommand{\corollaryname}{Corollary}
  \providecommand{\lemmaname}{Lemma}
  \providecommand{\propositionname}{Proposition}
\providecommand{\theoremname}{Theorem}

\def \Conv {\mathrm{Conv}}
\def \W {\mathrm{W}}
\def \W {\mathrm{W}}
\def \CC {\mathbf{C}}

\begin{document}

\title{Cone points of Brownian motion in arbitrary dimension}

\author{Yotam Alexander and Ronen Eldan}
\maketitle
\begin{abstract}
We show that the convex hull of the path of Brownian motion in $n$-dimensions, up to time $1$,
is a smooth set. As a consequence, we conclude that a Brownian motion
in any dimension almost surely has no cone points for any cone whose
dual cone is nontrivial. 
\end{abstract}

\section{Introduction}

Fix a dimension $n\geq2$. Let $B(\cdot)$ be a standard Brownian
motion in $\mathbb{R}^{n}$. Our main object of concern in this paper
will be 
\[
K=\Conv(\{B(t)| ~0\leq t\leq1\}).
\]
where $\Conv(\cdot)$ denotes the convex hull. \\

For a convex set $K$ and a point $x\in\partial K$, we say that
$x$ is \emph{singular} if the supporting hyperplane to $K$ at $x$ is not
unique. We say that $K$ is smooth if none of its boundary points
are singular. The main result of this paper is the following: 
\begin{thm} \label{thm:main}
$K$ is smooth almost surely. 
\end{thm}

In two dimensions, the fact that the boundary of $K$ is $C_{1}$-smooth
was first stated by Paul L\'evy in 1948 \cite{LE}, and was later rigorously
established in \cite{ELB} (see also \cite{CHM}). \\

When $n=2$, we say that $x$ is a $\alpha$-cone point of the two
dimensional Brownian motion $B(t),~ 0\leq t\leq1$, if there exists
$t_{0}\in[0,1]$ and $\delta>0$ such that $B(t_{0})=x$ and such that $B(t)\in W, ~ \forall t\in[t_{0}-\delta,t_{0}+\delta]$
where $W$ is a wedge whose tip lies at $x$ and with angle $\alpha$. In two dimensions, the
smoothness of $K$ is related to the absence of cone points of angle
smaller than $\pi$. Furthermore, it was proven in \cite{EV} that,
almost surely, the Hausdorff dimension of the set of $\alpha$-cone
points is equal to $2-\frac{2\pi}{\alpha}$. \\

The objective of the present paper is to initiate the investigation
of cone points in higher dimension. By taking a countable intersection
over rational $\delta$, we can strengthen the aforementioned smoothness,
proving that cone points in which the supporting cone is strictly
contained in a half-space do not exist in any dimension: 
\begin{cor}
Almost surely, there does not exists $t_{0}\in[0,1]$, $\delta>0$
such that $\bigl\{ B(t),|t-t_{0}|\leq\delta\bigr\}$ is contained
in a convex cone $\mathcal{C}$ whose tip lies at $B(t_{0})$ and which
is strictly contained in a half-space. 
\end{cor}

\begin{proof}
For any $t_{1}<t_{2}$, define $K[t_{1},t_{2}]:=\Conv(\{B(t),t_{1}\leq t\leq t_{2}\})$.
Because of Brownian scaling and the strong Markov property, Theorem
1 implies that $K[t_{1},t_{2}]-B(t_{1})$ (and hence also $K[t_{1},t_{2}]$
itself) is almost surely smooth. Taking a countable intersection over
all rational $t_{1},t_{2}$, we get that almost surely $K[t_{1},t_{2}]$
is smooth for all $t_{1},t_{2}\in[0,1]\cap\mathbb{Q}$. For every
$t_{0}\in[0,1]$, $\delta>0$, there exist a pair of rational numbers
$q_{1},q_{2}$ such that $t_{0}\in[q_{1},q_{2}]$ and $[q_{1},q_{2}]\subseteq\bigl\{|t-t_{0}|\leq\delta\bigr\}$.
The smoothness of $K[q_{1},q_{2}]$ implies that $\bigl\{ B(t),|t-t_{0}|\leq\delta\bigr\}$
cannot be contained in a cone whose tip lies at $B(t_{0})$ which
is strictly contained in a halfspace. Indeed, such a cone would necessarily also be contained in the intersection of two distinct halspaces. Thus, either $B(t_0)$ is in the interior of $K[q_1,q_2]$ or if it on the boundary, that would contradict the uniqueness of the supporting hyperplanes to $K[q_{1},q_{2}]$
at $B(t_{0})$.
\end{proof}

In the planer case ($n=2$), the smoothness of $K$ can be reduced to a countable intersection of events in the following way: For every pair of rational directions, consider the event that the supporting hyperplanes to $K$ in these directions coincide at a point of the Brownian motion. If $K$ is not smooth, there must exist two rational directions whose respective event occurs. This allows us to reduce the smoothness to bounds concerning the local behavior of Brownian motion, which in turn relies on the decomposition of the Brownian motion to two independent coordinates and several classical bounds regarding the probability of a one dimensional Brownian motion to be contained in a small interval. \\

Dimensions higher than $2$ seem to pose a significant additional difficulty: unlike the planar case, it appears that one is not able to express the smoothness
of the convex hull as the intersection of a countable family of \textquotedbl{}local\textquotedbl{}
events. To put it differently, in order for a boundary point of the
convex hull in three dimensions to be smooth, one needs to check that
the convex hull is not contained in any wedge among a one-parameter
family of wedges, and this event cannot be written as an intersection
of a countable number of events that only depend on a two-dimensional
behavior. \\

Alternatively, the smoothness of a convex body amounts to the smoothness
of every 2-dimensional projection, however there is no hope of reducing
the smoothness of $K$ to the smoothness of a countable set of two-dimensional
projection. To illustrate this, consider the set $\{(x_{1},x_{2},x_{3});x_{3}>x_{1}^{2}+|x_{2}|\}$.
This set is not smooth, while for all directions except a set of measure
zero, the corresponding two-dimensional projection is smooth. Therefore
there does not seem to be a way to take advantage of the fact that
every fixed two-dimensional projection of $K$ is smooth with probability
one.

\section{An approximating polytope}

A key step in our proof is to consider an approximation of $K$ by
polytopes defined as the convex hull of an associated random walk
partial to the range of $B(t)$. In section \ref{sec:reduction} below, we will see how to reduce the smoothness of $K$ to quantitative behavior of the facets of those discrete approximations. In turn, the behavior of those facets can be made tractable via a formula derived in Section \ref{sec:formulafacets}. The construction as well as some of the formulas that make it accessible were used in \cite{ELD} in order to compute several quantities related to $K$, such as its volume and its surface area. \\

We construct the random walk as follows. Let $P=((x_{1},y_{1}),(x_{2},y_{2}),...)$ be a Poisson point process
of intensity $1$, independent of $B(\cdot)$, in the set $[0,1]\times[0,\infty]$
and for all $\alpha\geq0$, define 
\[
\Lambda_{\alpha}=\left\{ x_{i}|~y_{i}\leq\alpha,~i\in\mathbb{N}\right\} \cup\{0,1\}.
\]
The process $\Lambda$ can be thought of as a \textquotedbl{}Poisson
rain\textquotedbl{} on the interval $[0,1]$: note that for all $\alpha\geq0$,
$\Lambda_{\alpha}$ is a Poisson point-process of intensity $\alpha$
on the unit interval and that the family $\Lambda_{\alpha}$ is increasing
with $\alpha$. For a fixed value of $\alpha$, writing $\Lambda_{\alpha}=(t_{1},...,t_{N})$
where $0=t_{1}\leq...\leq t_{N}=1$, we can think of $(B(t_{1}),B(t_{2}),...,B(t_{N}))$
as a random walk in $\mathbb{R}^{n}$. Finally, for all $\alpha>0$,
we define 
\[
K_{\alpha}=\Conv(\{B(t)|~t\in\Lambda_{\alpha}\}),
\]
so $K_{\alpha}$ is a monotone sequence of discrete approximations
of $K$.

\subsection{A formula for the facets} \label{sec:formulafacets}

In this section, we recall some notions from \cite{ELD}, towards a formula which allows us to calculate the expectation
of quantities related to the facets of $K_\alpha$. We begin with some notation. Let $\Delta_{n}$ be the $n$-dimensional simplex,
namely 
\[
\Delta_{n}=\left\{ (r_{1},...,r_{n})\in[0,1]^{n};~r_{1} < \dots < r_{n}\right\} .
\]
Next, for $r\in\Delta_{n}$ we define (by slight abuse of notation)
\[
B(r):=\{B(r_{1}),...,B(r_{n})\},
\]
and $F_{r}:=\Conv\bigl(B(r)\bigr)$ which is almost surely an $(n-1)$-dimensional
simplex. Let $n_{r}$ be a unit vector normal to $F_{r}$ chosen such
that $\langle n_{r},B(r_{1})\rangle\geq0$.

Next, for a Borel subset $A\subset\Delta_{n}$ we define 
\[
q_{\alpha}(A)=\#\{r\in A;~F_{r}\mbox{ is a facet in the boundary of }K_{\alpha}\}.
\]
We also need the definition of the point process 
\[
w_{\alpha}(A)=\#\left\{ r\in A;~~\{r_{1},..,r_{n}\}\subset\Lambda_{\alpha}\right\} 
\]
which we can think of as points $r\in\Delta_{n}$ which are candidates
to be facets of $K_{\alpha}$ in the sense that all their vertices
points of the random walk. Moreover, we define the (deterministic) measures
\[
\mu_{\alpha}(\cdot)=\mathbb{E}\left[q_{\alpha}(\cdot)\right]\mbox{ and }\nu_{\alpha}(\cdot)=\mathbb{E}[w_{\alpha}(\cdot)].
\]

Let $\mathcal{F}_{B}$ be the $\sigma$-algebra generated by the Brownian
motion $B(\cdot)$ (so that a random variable is measurable with respect
to $\mathcal{F}_{B}$ if and only if it does not depend on the point
process $\Lambda$). Let $f:\Delta_{n}\to\mathbb{R}$ be a function
such that for all $r\in\Delta_{n}$, $f(r)$ is a random variable
which is measurable with respect to $\mathcal{F}_{B}$. Then by a
calculation using the Campbell-Little-Mecke formula, one obtains (see \cite[Equation (3.7)]{ELD})
\[
\mathbb{E}\left[\int_{\Delta_{n}}f(r)dq_{\alpha}(r)\right]=\alpha^{n}\int_{\Delta_{n}}\mathbb{E}\left[f(r)\mathbf{1}_{E_{\alpha}(r)}\right]dr,
\]
where 
\[
E_{\alpha}(r):=\left\{ F_{r}\mbox{ is a facet in the boundary of }\Conv\left(B(r)\cup K_{\alpha}\right)\right\} .
\]

Following the exact same lines, this formula can be generalized in
the following sense: Let $f:\Delta_{n}\times\Delta_{n}\to\mathbb{R}$
be a function such that for all $(r,s)\in\Delta_{n}\times\Delta_{n}$,
$f(r,s)$ is a random variable which is measurable with respect to
$\mathcal{F}_{B}$. Define 
\[
E_{\alpha}(r,s):=\left\{ F_{r},F_{s}\mbox{ are facets in the boundary of }\Conv\left(B(r)\cup B(s)\cup K_{\alpha}\right)\right\} .
\]
Then, the generalized formula reads
$$
\mathbb{E}\left[\int_{\Delta_{n}\times\Delta_{n}}f(r,s)dq_{\alpha}(r)dq_{\alpha}(s)\right]= \alpha^{2n}\int_{\Delta_{n}\times\Delta_{n}}\mathbb{E}\left[f(r,s)\mathbf{1}_{E_{\alpha}(r,s)}\right]drds.
$$
Since $E_{\alpha}(r,s)\subseteq E_{\alpha}(r)\cap E_{\alpha}(s)$, we finally have
\begin{equation}\label{eq:formula1}
\mathbb{E}\left[\int_{\Delta_{n}\times\Delta_{n}}f(r,s)dq_{\alpha}(r)dq_{\alpha}(s)\right] \leq \alpha^{2n}\int_{\Delta_{n}\times\Delta_{n}}\mathbb{E}\left[f(r,s)\mathbf{1}_{E_{\alpha}(r)\cap E_{\alpha}(s)}\right]drds. 
\end{equation}

\subsection{Reducing the smoothness of $K$ to an approximate smoothness of $K_{\alpha}$} \label{sec:reduction}

In this section lies the idea behind the definition of the approximating polytope, $K_\alpha$. We will show that the non-smoothness of $K$ amounts, roughly, to the following asymptotic behavior of $K_\alpha$, as $\alpha \to \infty$: Given that $K$ is not smooth, for sufficiently large $\alpha$, the polytope $K_\alpha$ has two \emph{discordant facets}, namely, facets with distance of order $1/\sqrt{\alpha}$ from each other, such that the angle between the corresponding normal directions is bounded away from zero. In the upcoming subsections, our main goal will be to show that this is not possible. \\

Let us first introduce some notation. For $(r,s)\in\Delta_{n}\times\Delta_{n}$,
define $L(r,s)$ to be the intersection of the $(n-1)$-dimensional
affine subspaces spanned by $F_{r}$ and $F_{s}$. Moreover, define
\[
W(r,s):=\left\{ x+y;x\in L(r,s),\langle y,n_{r}\rangle\leq0\mbox{ and }\langle y,n_{s}\rangle\leq0\right\} ,
\]
the wedge defined by the facets $F_{r},F_{s}$ and set $\theta(r,s)=\arccos(\langle n_{r},n_{s}\rangle)$,
the inner angle of the wedge $W(r,s)$. \\
A key definition for us will be the event
\[
\CC_{\alpha,\gamma,\theta}(r,s):= \left\{ \theta(r,s)\geq\theta\right\} \cap\{\mathrm{dist}(F_{r}\cup F_{s},L(r,s))\leq\gamma\}.
\]
where $\alpha,\theta,\gamma > 0$. When the event $\CC_{\alpha,\gamma,\theta}(r,s)$ holds, we will say that $F_s$ and $F_r$ are discordant facets. Now let 
\[
\phi(\alpha):=e^{\sqrt{\log\alpha}}.
\]
Define the event 
\[
N_{\alpha}[a,b]=\left\{ \forall t\in[a,b],\Lambda_{\alpha}\cap[t-\alpha^{-1}\phi(\alpha),t+\alpha^{-1}\phi(\alpha)]\neq\emptyset\right\} ,
\]
and denote $N_{\alpha}:=N_{\alpha}\left[0,1\right]$. Next, define
$M_\delta[a,b]:=\sup_{t_{1},t_{2}\in\left[a,b\right],\!|t_{1}-t_{2}|\leq\delta}|B(t_{1})-B(t_{2})|$,
denote $M_\delta:=M_\delta[0,1]$ and consider the events
\[
Y_{\alpha}[a,b]=\left\{ M_\delta[a,b]\leq\delta^{1/2}\phi(\alpha)+\alpha^{-2n-1},~~\forall\delta\leq b-a\right\} ,
\]
and
\[
R_{\alpha}[a,b]:=N_{\alpha}[a,b]\cap Y_{\alpha}[a,b],\quad R_{\alpha}:=R_{\alpha}[0,1].
\]
The following lemma, whose proof is postponed to the appendix, is based on standard estimates.
\begin{lem} \label{lem:Ralpha}
For every dimension $n\geq2$, there exists $C>0$ such that for all $\alpha > C$ we have 
\begin{equation}
\mathbb{P}\big(R_{\alpha}^{C}\big) \leq \alpha^{-2n-1}. \label{eq:pra}
\end{equation}
\end{lem}

Remark that for $\alpha$ larger than some universal constant, we
have the following implication:

\begin{equation}
R_{\alpha}\mbox{ holds}\Rightarrow\forall t\in[0,1],~\exists s\in\Lambda_{\alpha}\mbox{ such that }|B(s)-B(t)|\leq\frac{\phi(\alpha)^{2}}{\sqrt{\alpha}}.\label{eq:RWdense}
\end{equation}
Next, we formulate a geometric lemma which will allow us to relate between the
event $\CC_{\alpha,\gamma,\theta}(r,s)$ and the smoothness of $K$: 
\begin{lem} \label{lem:annoying}
For every $\pi>\kappa>0$ there exists $M_\kappa>0$ such that the following
holds: Let $W$ be a wedge whose tip contains the origin and whose
opening angle is $\pi-\kappa$. Let $P$ be a convex polytope contained
in $W$ whose distance from the origin is at most $s$. Then
there exist two facets $F_{1},F_{2}$ of $P$ with normal directions
$n_{1},n_{2}$ such that the angle between $n_{1}$ and $n_{2}$ is
at least $\tfrac{\kappa}{16}$ and such that the following holds:
let $L$ be the span of $n_{1},n_{2}$ and let $W'\subset L$ be the
wedge corresponding to $F_{1},F_{2}$ with tip $w'$, then the distance
between $w'$ and the projection of $F_{1}$ on $L$ is at most $M_\kappa s$. 
\end{lem}

The proof of this lemma is postponed to the end of the section. By slight abuse of notation, we now define
$$
\CC_{\alpha, \kappa}(r,s) := \CC_{\alpha, \gamma(\alpha,\kappa), \kappa/16}(r,s)
$$
where
$$
\gamma(\alpha, \kappa)=M_\kappa \frac{\varphi^{2}(\alpha)}{\sqrt{\alpha}}
$$
and where $M_\kappa$ is the constant given by the above lemma. As a direct corollary of this lemma, we have the following proposition: 
\begin{prop}
	For every $\kappa>0$ there exist constants $C_\kappa,M_\kappa>0$ such that the
	following holds almost surely: For any $\alpha>C_\kappa$, suppose that there exists $x\in\partial K$
	such that $K$ is contained in a wedge of angle $\kappa$ whose tip
	is at $x$ and that $R_{\alpha}$ holds. Then there
	exist $(r,s)\in\triangle_{n}\times\triangle_{n}$ such that $F_r, F_s$ are facets on the boundary of $K_\alpha$ and such that the event 
	$\CC_{\alpha, \kappa}(r,s)$ holds.
\end{prop}
\begin{proof}
Using the implication \eqref{eq:RWdense}, we can invoke Lemma \ref{lem:annoying} with $s=\frac{\varphi^{2}(\alpha)}{\sqrt{\alpha}}$ for the wedge spanned by the supporting hyperplanes to $K$ at $x$, and with the polytope being $K_\alpha$. The existence of the facets $F_1,F_2$ amounts to the required events for some $r,s$.
\end{proof}

A consequence of the above proposition is that whenever $K$ is not smooth, then necessarily, for all $\alpha$ larger than some constant, either the event $R_\alpha^C$ holds, or otherwise, one has that for some $\kappa > 0$, 
$$
\int_{\Delta_{n}\times\Delta_{n}}\mathbf{1}_{\CC_{\alpha,\kappa}(r,s)}dq_{\alpha}(r)dq_{\alpha}(s) \geq 1.
$$
Using Markov's inequality, our main theorem will thus follow if we prove that for every $\kappa>0$, one has 
$$
\mathbb{P}(R_\alpha^C) + \mathbb{E}\left[\int_{\Delta_{n}\times\Delta_{n}}\mathbf{1}_{\CC_{\alpha,\kappa}(r,s)}dq_{\alpha}(r)dq_{\alpha}(s)\right]\xrightarrow{\alpha\to\infty}0.\label{eq:enoughtoshow1}
$$
We are now in position to apply Equation \eqref{eq:formula1}, according to which it is enough to show that

\begin{equation}\label{eq:enoughtoshow3}
\mathbb{P}(R_\alpha^C) + \alpha^{2n}\int_{\Delta_{n}\times\Delta_{n}}\mathbb{E}\left[\mathbf{1}_{\CC_{\alpha,\kappa}(r,s)}\mathbf{1}_{E_{\alpha}(r)\cap E_{\alpha}(s)}\right]drds \xrightarrow{\alpha\to\infty}0.
\end{equation}
At this point, it will be more convenient to work with the following events:
\[
\tilde{E}_{\alpha}(r):=\left\{ \forall t\in[0,1],\langle B(t),n_{r}\rangle\leq\langle B(r_{1}),n_{r}\rangle + \frac{\phi(\alpha)^{2}}{\sqrt{\alpha}}\right\} .
\]
This event is almost similar to $E_{\alpha}(r)$ in the following
sense: Equation (\ref{eq:RWdense}) implies that 
\begin{equation}
R_{\alpha}\cap E_{\alpha}(r)\subset\tilde{E}_{\alpha}(r).\label{eq:expandedfacet}
\end{equation}
Finally, we consider the event 
\[
\tilde{\CC}_{\alpha,\kappa}(r,s):=\tilde{E}_{\alpha}(r)\cap\tilde{E}_{\alpha}(s)\cap \CC_{\alpha, \kappa}
\]
Using (\ref{eq:expandedfacet}), we now have that Equation (\ref{eq:enoughtoshow3}) follows from 
\begin{equation}
\alpha^{2n} \mathbb{P}\left(R_{\alpha}^{C}\right)+\alpha^{2n}\int_{\Delta_{n}\times\Delta_{n}}\mathbb{P}\left(\tilde{\CC}_{\alpha,\kappa}(r,s)\right)drds\xrightarrow{\alpha\to\infty}0.\label{eq:enoughtoshow2}
\end{equation}
The proof of the main theorem now boils down to proving the last equation.
Lemma \eqref{lem:Ralpha} implies that $\alpha^{2n} \mathbb{P}\left(R_{\alpha}^{C}\right)\xrightarrow{\alpha\to\infty}0$
, and the rest of the paper is devoted to estimating the integral.
The content of this section is summarized by the following statement: 
\begin{prop} \label{prop:reduction}
	For every dimension $n$, one has
	\[
	\mathbb{P}(K\mbox{ is not smooth})\leq\sup_{\kappa>0}\limsup_{\alpha\to\infty}\alpha^{2n}\int_{\Delta_{n}\times\Delta_{n}}\mathbb{P}\left(\tilde{\CC}_{\alpha,\kappa}(r,s)\right)drds.
	\]
\end{prop}

The rest of the proof is devoted to finding an upper bound for the quantities on the right hand side.

\begin{proof}[Proof of Lemma \ref{lem:annoying}]
	Without loss of generality, we may assume that $\kappa\leq\frac{\pi}{2}$
	(indeed, otherwise we may consider a wedge with a bigger opening angle
	that contains the wedge $W$). Let $u_{1},u_{2}$ be the inner normal directions to the
	facets of $W$. Define $h=\frac{u_1+u_2}{|u_1+u_2|}$. Since, by assumption, $P$ contains a point whose 
	distance to the origin is at most $s$, then $P$ must contain a vertex $v$ such that $\langle v, h \rangle \leq s$. Since
	$v$ is inside the triangular prism defined as the intersection of $W$ with $\{x; \langle x, h \rangle \leq s \}$, the distance between $v$ and the tip
	of the wedge $W$ is bounded by $s' := \frac{s}{\sin(\kappa/2)}$.
	
	Let $F$ be an $(n-1)$-dimensional facet containing the vertex $v$. Denote by $u$ the inner normal
	to $P$ at $F$.  We claim that there exists $i\in\{1,2\}$ such that
	the angle between $u$ and $u_{i}$ is within the range $[\tfrac{\kappa}{2},\pi-\tfrac{\kappa}{2}]$.
	Indeed, since the angle between $u_{1}$ and $u_{2}$ is $\kappa$
	which is at most $\pi/2$, by the triangle inequality it cannot be
	the case that both vectors have angle less than $\kappa/2$ with either
	the vector $u$ or with its antipodal. Assume without loss of generality
	that the vector $u_{1}$ satisfies the above, hence 
	\begin{equation}
	\arccos(\langle u,u_{1}\rangle)\in[\tfrac{\kappa}{2},\pi-\tfrac{\kappa}{2}]\label{eq:assumpangle}
	\end{equation}
	(otherwise we may switch between $u_{1}$ and $u_{2}$).
	
	Consider the plane $H$ spanned by $u$ and $u_{1}$, passing at $v$.
	Set $P'=P\cap H$. The set $P'$ is a (possibly degenerate) $2$-dimensional
	convex polygon. First suppose that $P'$ has an empty interior. In
	this case, there must exist another facet $F'$ of $P$ which contains
	the vertex $v$ and whose normal direction has a nonpositive scalar
	product with $u$. In this case the lemma is complete by considering
	the facets $F,F'$. Otherwise, we may assume that $P'$ is a convex
	polygon with nonempty interior. Denote its edges by $f_{1},f_{2},...$
	in a manner that $f_{i}$ and $f_{i+1}$ share a vertex and such that
	$f_{1}$ contains the vertex $v$. Each edge $f_{i}$ corresponds
	to a facet $F_{i}$ of $P$, whose normal we denote by $n_{i}$.
	
	For all $t\in\mathbb{R}$ let $j(t)$ be the curve starting at $v$
	going along $P'$ in a speed parametrized by arclength, the direction
	of which will be chosen promptly out of the two possible directions.
	We define 
	\[
	g(t):=\langle j(t),u_{1}\rangle.
	\]
	The assumption \eqref{eq:assumpangle} implies that for a suitable
	choice of direction one has that 
	\begin{equation}
	\min(g'_{+}(0),g'_{-}(0))\leq-\sin(\kappa/2)\label{eq:derg1}
	\end{equation}
	(here $g'_{+}(0),g_{-}'(0)$ denote the right and left-derivatives
	of $g$ at $0$). Define $\ell=\frac{s'}{\sin(\kappa/4)}$. Since $g(0)\leq s$
	and $g(t)\geq0$ for all $t$, by the fact that $g$ is piecewise-differentiable
	it follows that there exists $t'\leq\ell$ such that 
	\[
	g'(t')\geq-\sin(\kappa/4).
	\]
	The above inequality combined with \eqref{eq:derg1} imply that the
	angle between $f_{1}$ and $f_{i}$ is at least $\kappa/4$, where
	$f_{i}$ is the facet containing the point $j(t')$. Since $u$ parallel
	to $H$, it follows that the angle between $u$ and $n_{i}$ is at
	least $\kappa/4$. Note also that the distance between $F$ and $F_{i}$
	is by definition at most $\frac{s'}{\sin(\kappa/4)}$ and the same
	is true for the distance between $F_{k},F_{k'}$ for all $1\leq k<k'\leq i$.
	
	At this point, the lemma will be concluded given that we find two
	indices $1\leq k<k'\leq i$ such that at least one of the following
	holds: (i) the angle between $n_{k}$ and $n_{k'}$ is in the rangle
	$[\tfrac{\kappa}{8},\pi-\tfrac{\kappa}{8}]$ or (ii) we have $k'=k+1$
	and the angle between $n_{k}$ and $n_{k'}$ is at least $\tfrac{\kappa}{8}$.
	This would be enough to complete the proof since if we consider the
	wedge $W'$ spanned by the two facets $F_{k},F_{k'}$ then we have
	that the angle between the facets is at least $\kappa/8$ in both
	cases, moreover if $W'$ is the wedge spanned by those facets with
	tip $w'$, then in case (ii) the distance of both facets to the tip
	is $0$ and in case (i) the distance of each facet to the tip is within
	factor $\frac{1}{\sin(\kappa/8)}$ of the distance between the facets.
	
	Suppose now that case (ii) does not hold, hence that the angle between
	$F_{k}$ and $F_{k+1}$ is at most $\kappa/8$ for all $1\leq k\leq i-1$.
	In this case either there exists $1\leq k'\leq i$ such that the angle
	between $F_{k'}$ and the plane $H$ is at least $\kappa/8$ (in which
	case we can take $k=1$ and case (i) follows), or otherwise it must
	be that the angles between $f_{k}$ and $f_{k+1}$ are at most $\kappa/4$
	for all $1\leq k\leq i-1$. Since the angle between $f_{1}$ and $f_{i}$
	is at least $\kappa/4$ and since $\kappa\leq\pi/2$, it follows by
	a continuity argument that there exists some $k'\leq i$ such that
	the angle between $f_{1}$ and $f_{k}$ within the range $[\tfrac{\kappa}{4},\pi-\tfrac{\kappa}{4}]$
	and case (i) holds again. This completes the proof. 
\end{proof}

\section{An upper bound on the probability for discordant facets}

For $(r,s)\in\Delta_{n}\times\Delta_{n}$, we denote by $\mathbf{t}(r,s)=(t_{1},...,t_{2n})$
the points of $r\cup s$ in increasing order. 

The main proposition of this section is the following one: 
\begin{prop} \label{prop:tildec}
	For every dimension $n$ and every $\kappa>0$ there exists $C>0$
	such that the following holds, for all $\alpha>C$: Fix $(r,s)\in\Delta_{n}\times\Delta_{n}$.
	Denote $(t_{1},...,t_{2n})=\mathbf{t}(r,s)$. Then one has, 
	\[
	\mathbb{P}\left(\tilde{\CC}_{\alpha,\kappa}(r,s)\right)\leq\alpha^{-2n-1}+\alpha^{-2n -\kappa/(16000n)}\frac{1}{\sqrt{t_{1}(1-t_{2n})}}\prod_{i=2}^{2n}\frac{1}{t_{i}-t_{i-1}}.
	\]
\end{prop}

By combining this bound with Proposition \ref{prop:reduction}, the proof of our main theorem boils down to estimating an explicit integral over $\Delta_n \times \Delta_n$. \\

Throughout the section, we fix $\kappa > 0$ and $(r,s)\in\Delta_{n}\times\Delta_{n}$
and the corresponding times $(t_{1},...,t_{2n})=\mathbf{t}(r,s)$. 
We also fix points $b_{0}=0$, $b_{1},...,b_{2n+1}\in\mathbb{R}^{n}$,
and consider the event 
\begin{equation}
S:=\bigl\{ B(t_{i})=b_{i},~~\forall1\leq i\leq2n+1\bigr\}.\label{eq:defS}
\end{equation}
Consider the following two properties: 
\begin{equation} \label{eq:prop1}
|b_{i+1}-b_{i}|\leq\phi(\alpha)\sqrt{t_{i+1}-t_{i}}+\alpha^{-2n-1}.
\end{equation}
\begin{equation} \label{eq:prop2}
\left\{ \theta(r,s)\geq\kappa/16\right\} \cap\{\mathrm{dist}(F_{r}\cup F_{s},L(r,s))\leq\gamma(\alpha,\kappa)\}.
\end{equation}

Our goal will be to replace the event $\tilde{\CC}_{\alpha,\kappa}(r,s)\cap R_{\alpha}$
with an intersection of events which are independent upon conditioning on $S$. To that end,
for all $i$, $0\leq i\leq2n$, we write 
\begin{equation}\label{eq:defHi}
H_{i}:=\Bigl\{\forall t\in[t_{i},t_{i+1}],~~\langle B(t),n_{r}\rangle\leq\langle B(r_{1}),n_{r}\rangle+\tfrac{\phi(\alpha)^{2}}{\sqrt{\alpha}} \mbox{ and }\langle B(t),n_{s}\rangle\leq\langle B(s_{1}),n_{s}\rangle+\tfrac{\phi(\alpha)^{2}}{\sqrt{\alpha}} \Bigr\},
\end{equation}
with the convention $t_{0}=0$ and $t_{2n+1}=1$. Remark that
$$
\tilde{E}_{\alpha}(r) \cap \tilde{E}_{\alpha}(s) = \bigcap_{i=0}^{2n} H_i,
$$
and by the representation theorem for the Brownian bridge, we have that the events $H_{i}$ are independent conditioned on $S$. Thus,
\begin{align}\label{HiProd}
\left . \mathbb{P}\left(\tilde{E}_{\alpha}(r) \cap \tilde{E}_{\alpha}(s) \cap R_{\alpha} \right |S\right) ~& \leq \left . \mathbb{P}\left( \tilde{E}_{\alpha}(r) \cap \tilde{E}_{\alpha}(s)\cap \left (\bigcap_{i=0}^{2n}R_{\alpha}[t_{i},t_{i+1}] \right ) \right |S\right) \\
& = \left . \mathbb{P}\left(\bigcap_{i=0}^{2n} \left (H_i \cap  R_{\alpha}[t_{i},t_{i+1}]\right ) \right |S\right) \nonumber \\
& = \prod_{i=0}^{2n}\mathbb{P} \left . \bigl (H_{i} \cap R_{\alpha}[t_{i},t_{i+1}] \right |S \bigr ). \nonumber
\end{align}
The following lemma shows that properties \eqref{eq:prop1} and \eqref{eq:prop2} imply
the existence of a ``special'' interval $\left[t_{j},t_{j+1}\right]$
among $\left[t_{0},t_{1}\right],\left[t_{1},t_{2}\right],...,\left[t_{2n},t_{2n+1}\right]$, which roughly has the property that the starting point of associated Brownian bridge has distance to the tip of the wedge which is much smaller than the square root of the length of the time interval. For this interval, we will be able to derive an improved bound on $\mathbb{P}(H_{j}|S)$ in the next section.

\begin{lem} \label{lem:specialindex}
For every $M>0$ there exists $C>0$ such that for all $\alpha > C$, the following holds. Let $t_0=0, t_{2n+1}=1$ and $0 < t_1,...,t_{2n} < 1$. Let $b_0 = (0,0)$ and, for all $1 \leq i \leq 2n+1$, let $b_i \in \mathbb{R}^2$ so that the condition \eqref{eq:prop1} holds. Let $w_0 \in \mathbb{R}^2$, and suppose that there exists an index $0 \leq j_0 \leq 2n+1$ for which $|b_{j_0} - w_0| < M \frac{\phi(\alpha)^{2}}{\sqrt{\alpha}}$. Then there exists an index $0 \leq j \leq 2n$ such that 
\begin{equation}
t_{j+1}-t_{j} \geq \alpha^{1/(10n)}\max\left( \min(|b_j-w_{0}|, |b_{j+1}-w_0|)^{2},\frac{1}{\alpha}\right).
\end{equation}
\end{lem}
The proof of this lemma is postponed to the appendix. Denote by $P$ the orthogonal projection to $sp\left\{ n_{r},n_{s}\right\}$ and let $w_0$ be the tip of the wedge created by $P F_s$ and $P F_r$. Then condition \eqref{eq:prop2} implies that there exists an index $j_0$ for which $|P b_{j_0} - w_0| < M \frac{\phi(\alpha)^{2}}{\sqrt{\alpha}}$. We may now invoke the above lemma on the points $(P b_0, ..., P b_{2n})$ to conclude that whenever $\alpha > C$ where $C$ depends only on $n$, there exists an index $0 \leq j \leq 2n$ for which 
\begin{equation} \label{eq:condj}
t_{j+1}-t_{j}\geq\alpha^{1/(10n)}\max\left( \min(|P b_j-w_{0}|, |P b_{j+1} - w_0|)^{2},\frac{1}{\alpha}\right).
\end{equation}

The next proposition consists of the core estimates which will be combined in order to yield the bound of Proposition \ref{prop:tildec}. Its proof is based on two dimensional estimates for exit probabilities for a Brownian motion and Brownian Bridge, which will be established in the next section.
\begin{prop} \label{prop:bounds}
	For every dimension $n$ and every $\theta \geq 0$ and $\varepsilon>0$ there exists a constant
	$C_{\theta,\epsilon}$ such that the following holds. Let $(r,s) \in \Delta_n \times \Delta_n$ and let
	$(t_0,..,t_{2n+1}) = \mathbf{t}(r,s)$. Fix points $b_{1},\dots,b_{2n+1}\in\mathbb{R}^{n}$. Consider the events $S$ and $\{H_i\}_{i=0}^{2n}$ defined in equations (\ref{eq:defS}) and \eqref{eq:defHi}. Then almost surely, we have for all $\alpha > C_{\theta(r,s), \epsilon}$ the following bounds.
	\begin{enumerate}[(i)]
	\item
	For all $1\leq i\leq 2n-1$, we have 
	\begin{equation} 
	\mathbb{P}(H_{i} \cap R_\alpha[t_i, t_{i+1}]~|S)\leq\frac{1}{(t_{i+1}-t_{i})\alpha} \cdot \alpha^{\varepsilon}.\label{eq:boundH1}
	\end{equation}
	\item	
	For $i\in\{0,2n\}$, we have 
	\begin{equation}
	\mathbb{P}(H_{i}\cap R_\alpha[t_i, t_{i+1}]~|S)\leq\frac{1}{\sqrt{(t_{i+1}-t_{i})\alpha}} \cdot \alpha^{\varepsilon}.\label{eq:boundH2}
	\end{equation}
	\item
	\label{item:part3} For all $1\leq j \leq 2n-1$ such that condition \eqref{eq:condj} is satisfied, we have
	\begin{equation}
	\mathbb{P}(H_{j}\cap R_\alpha[t_j, t_{j+1}]~|S)\leq\frac{1}{(t_{j+1}-t_{j})\alpha} \cdot \alpha^{-\tfrac{\theta(r,s)}{800n} + \varepsilon}\label{eq:boundH3}
	\end{equation}
	\item
	For $j \in \{0, 2n \}$ such that condition \eqref{eq:condj} is satisfied, we have
	\begin{equation}
	\mathbb{P}(H_{j} \cap R_\alpha[t_j, t_{j+1}]~|S)\leq\frac{1}{\sqrt{(t_{j+1}-t_{j})\alpha}} \cdot \alpha^{-\tfrac{\theta(r,s)}{800n} + \varepsilon}.\label{eq:boundH4}
	\end{equation}
	\end{enumerate}
\end{prop}
The proof of this proposition is the objective of the next section. Given those bounds, we are finally ready to prove the main result of the section.
\begin{proof}[Proof of Proposition \ref{prop:tildec}]
	Defining $b_{i}=B(t_{i})$ for $1\leq i\leq2n+1$, let $E_{1},E_{2}$
	be the events that properties \eqref{eq:prop1} and \eqref{eq:prop2} hold, respectively. Note
	that 
	\begin{equation} \label{eq:idd1}
	\tilde{\CC}_{\alpha,\kappa}(r,s) = E_{2} \cap \tilde{E}_{\alpha}(r)\cap\tilde{E}_{\alpha}(s)
	\end{equation}
	and
	\begin{equation} \label{eq:idd2}
	R_{\alpha}\subset E_{1}.
	\end{equation}
	
	Remark that, since we know that, under $E_1 \cap E_2$, condition \eqref{eq:condj} is satisfied for at least one index $j$, we have that one of the terms in the product $\prod_{i=0}^{2n}\mathbb{P} \left . \bigl (H_i \right |S \bigr )$ can be bounded by the improved bounds given in equations \eqref{eq:boundH3} and \eqref{eq:boundH4}. Thus, invoking Proposition \ref{prop:bounds} with $\varepsilon=\tfrac{\kappa}{2^{15} n(2n+1)}$ gives that under the event $E_1 \cap E_2$ we have, almost surely,
	\begin{align}\label{eq:finalproduct}
	\prod_{i=0}^{2n}\mathbb{P} \left . \bigl (H_i \cap R_\alpha[t_i, t_{i+1}]~ \right | & B(t_{1}),...,B(t_{2n+1}) \bigr ) \\ & \leq \alpha^{-2n-\theta(r,s)/(800n)+(2n+1)\varepsilon}\frac{1}{\sqrt{t_{1}(1-t_{2n})}}\prod_{i=2}^{2n}\frac{1}{(t_{i}-t_{i-1})} \nonumber \\
	& \leq \alpha^{-2n-\kappa/(16000n)}\frac{1}{\sqrt{t_{1}(1-t_{2n})}}\prod_{i=2}^{2n}\frac{1}{(t_{i}-t_{i-1})}, \nonumber
	\end{align}
	for all $\alpha > C_\kappa$ (where the term $\alpha^{-\theta(r,s)/(800n)}$ appears thanks to existence of the index $j$ given by \eqref{eq:condj}). Thus, we can calculate
	\begin{align*}
	\mathbb{P}\left(\tilde{\CC}_{\alpha,\kappa}(r,s)\right)~ & =\mathbb{E}\left[\mathbb{P}\left(\tilde{\CC}_{\alpha,\kappa}(r,s)|B(t_{1}),...,B(t_{2n+1})\right)\right]\\
	& \stackrel{\eqref{eq:idd1}}{=} \mathbb{E}\left[\mathbb{P}\left(\tilde{E}_{\alpha}(r)\cap\tilde{E}_{\alpha}(s)|B(t_{1}),...,B(t_{2n+1})\right)\mathbf{1}_{E_{2}}\right]\\
	& \stackrel{\eqref{eq:idd2}}{\leq}\mathbb{P}(R_{\alpha}^{C})+\mathbb{E}\left[\mathbb{P}\left(\tilde{E}_{\alpha}(r)\cap\tilde{E}_{\alpha}(s) \cap R_\alpha |B(t_{1}),...,B(t_{2n+1})\right)\mathbf{1}_{E_{1}\cap E_{2}}\right]\\
	& \stackrel{\eqref{HiProd}}{\leq}\mathbb{P}(R_{\alpha}^{C})+\mathbb{E}\left[\mathbf{1}_{E_{1}\cap E_{2}}\prod_{i=0}^{2n}\mathbb{P}(H_{i} \cap R_\alpha[t_i, t_{i+1}]~ |B(t_{1}),...,B(t_{2n+1}))\right]\\
	& \stackrel{\eqref{eq:finalproduct}}{\leq}\mathbb{P}(R_{\alpha}^{C})+\alpha^{-2n-\kappa/(16000n)}\frac{1}{\sqrt{t_{1}(1-t_{2n})}}\prod_{i=2}^{2n}\frac{1}{(t_{i}-t_{i-1})}\\
	& \stackrel{\eqref{eq:pra}}{\leq}\alpha^{-2n-1}+\alpha^{-2n-\kappa/(16000n)}\frac{1}{\sqrt{t_{1}(1-t_{2n})}}\prod_{i=2}^{2n}\frac{1}{(t_{i}-t_{i-1})}.
	\end{align*}
	This completes the proof. 
\end{proof}

\section{Estimates for two-dimensional wedges}

The objective of this section is to prove Proposition \ref{prop:bounds}, which amounts to new estimates regarding exit probabilities of Brownian bridges from two dimensional wedges. The needed estimates must have two features that do not seem to be supported by existing bounds in the literature: They need not only to exploit the fact that the path stays inside a wedge, but also to fully exploit the fact that both endpoints are close to the boundary of the wedge, and they should be valid upon conditioning on both endpoints. \\

For $w_0 \in \mathbb{R}^2$ and $\alpha,\beta>0$ we define, 
\[
\W(w_{0},\beta)=\left\{ w_{0}+t(\cos x,\sin x);~t\in[0,\infty),|x|\leq\beta\right\}.
\]

The following lemma is a direct consequence of \cite[Theorem 2]{SP}. 
\begin{lem} \label{lem:Spitzer}
For every $\pi>\theta\geq0$ there exists $J_{\theta}>0$ such that
the following holds. Let $t,r>0$. Let $W=\W(0,\pi-\theta)$. Let $B(t)$
be a planar Brownian motion with starting point $B(0)=a$, a point
satisfying $a\in W$ and $|a|=r$. Then, 
\[
\mathbb{P}\big(B(s)\in W,~~\forall s\in[0,t]\big)\leq J_{\theta}\left(\frac{r}{\sqrt{t}}\right)^{1+\frac{\theta}{2\pi}}.
\]
\end{lem}

The next lemma bound is the main technical ingredient of this section. It gives an upper bound for the probability of a Brownian bridge, started close to the tip of a wedge, to remain inside the wedge.
\begin{lem} \label{lem:technicalwedge}
For any $0 \leq \theta < \pi$ and $\varepsilon>0$ there exists $C>0$
such that the following holds. Suppose that $\alpha>C$ and $r<1$
and that $W=W(w_{0},\pi-\theta)$ for some $w_0 \in \mathbb{R}^2$. Let $a,b\in \mathbb{R}^2$ be points with $|a-w_{0}|=r$
and $|b-a|<2\phi(\alpha)$. Let $X(t)$ be a Brownian bridge with
boundary conditions $X(0)=a$ and $X(1)=b$. Then, 
\begin{equation} \label{eq:techbound}
\mathbb{P}\left(X(t)\in W,~~\forall t\in[0,1]\right)\leq\alpha^{\varepsilon}max(\alpha^{-1},r)^{1+\frac{\theta}{20}}.
\end{equation}
\end{lem}

\begin{proof}
By applying a translation, we assume without loss of generality that $a=0$. Furthermore it suffices
to prove \eqref{eq:techbound} for $r\geq\alpha^{-1}$. Indeed , it is easy to see
that the left hand side of \eqref{eq:techbound} is monotonically increasing with respect to $r$ (by monotonicity of the corresponding events). Define the event 
\[
E:=\left\{ X(t)\in W,~~\forall t\in[0,1]\right\} .
\]
We can write 
\begin{equation}
X(t)=B(t)-t(B(1)-b)\label{eq:XB}
\end{equation}
where $B(t)$ is a standard Brownian motion. Consider the event 
\[
F:= \bigl \{|B(t)|<t^{1/2}\phi(\alpha)+\alpha^{-2n-1},\forall t\in[0,1] \bigr \}
\]
Using Lemma \ref{lem:Ralpha} we have $\mathbb{P}(F)\geq1- \alpha^{-2n-1}$ whenever $\alpha$ is bigger than some universal constant. Let
$k$ be an integer whose value will be chosen later on. Define 
\[
u=r\phi(\alpha)^{-1},~t_{0}=0,~t_{j}=u^{2^{-(j-1)}},~\forall j\leq k.
\]
Next, consider the events 
\[
E_{j}:=\left\{ d(B(t),W)\leq4\phi(\alpha)r^{2^{-(j-1)}},~~\forall t\in[t_{j-1},t_{j}]\right\} .
\]
For each $1\leq j\leq k$, let
$$
w_j := w_0 - \left (4\sin \left (\frac{\pi-\theta}{2} \right )^{-1}\phi(\alpha)r^{2^{-(j-1)}}, ~~ 0 \right ) ~~\mbox{ and }~~ W_{j} := \W \left (w_j, \pi - \theta \right ),
$$
so that $W_j$ is a wedge with containing $W$, with the same opening angle, and whose tip $w_{j}$ is situated
``behind'' $w_{0}$ such that the distance between their outer boundaries
is $4\phi(\alpha)r^{2^{-(j-1)}}$. Clearly, we have 
\[
E_{j}\subseteq\left\{ B(t)\in W_{j},~~\forall t\in[t_{j-1},t_{j}]\right\} .
\]
Next consider the events
\[
E_{j}':=\left\{ |B(t_{j-1})|\leq\phi(\alpha)\sqrt{t_{j-1}}+\alpha^{-2n-1}\right\} \cap\left\{ B(t)\in W_{j},~~\forall t\in[t_{j-1},t_{j}]\right\} .
\]
First we claim that 
\begin{equation} 
E\subset\left(\bigcap_{j=1}^{k}E_{j}\right)\cup F^{C}\subseteq\left(\bigcap_{j=1}^{k}E_{j}'\right)\cup F^{C}.\label{eq:eej}
\end{equation}
Indeed, under the event $F$ we have $|B(1)|\leq\phi(\alpha)+\alpha^{-2n-1}\leq2\phi(\alpha)$,
and according to equation (\ref{eq:XB}), we have 
\[
|X(t)-B(t)|\leq4\phi(\alpha)t_{j}\leq4\phi(\alpha)r^{2^{-(j-1)}},~~\forall t<t_{j}.
\]
By the triangle inequality we have that $X(t)\in W$ implies the event
$E_{j}$. Under the event $F$ the event $E_{j}$ implies the event
$E_{j}'$, which establishes (\ref{eq:eej}). Next, by the Markov
property of Brownian motion, note that 
\[
\mathbb{P}\big(E_{j}'|~\sigma(E_{1}',...,E_{j-1}',B(t_{j-1}))\big)=\mathbb{P}\big(E_{j}'|~B(t_{j-1})\big).
\]
Remark that, under the event $E_{j}'$, we have
\begin{align}\label{eq:triang}
|B(t_{j-1})-w_{j}| ~& \leq |a- w_0| + |B(t_{j-1}) - a| +  |w_0 - w_{j}| \\
& \leq r + \phi(\alpha)\sqrt{t_{j-1}}+\alpha^{-2n-1} + 4\sin \left (\frac{\pi-\theta}{2} \right )^{-1}\phi(\alpha)r^{2^{-(j-1)}} \nonumber \\
& \leq 7\sin(\frac{\pi-\theta}{2})^{-1}\phi(\alpha)r^{2^{-(j-1)}}. \nonumber
\end{align}

We now apply Lemma \ref{lem:Spitzer} with respect to the enlarged wedge
$W_{j}$: 
\begin{align*}
\mathbb{P}\big(E_{j}'|~B(t_{j-1})\big) ~& \leq J_{\theta}\left(\frac{\mid B(t_{j-1})-w_{j}\mid}{\sqrt{t_{j}-t_{j-1}}}\right)^{1+\frac{\theta}{2\pi}} \\
& \stackrel{\eqref{eq:triang}}{\leq} J_{\theta}\left(\frac{7\sin(\frac{\pi-\theta}{2})^{-1}\phi(\alpha)r^{2^{-(j-1)}}}{\sqrt{t_{j}-t_{j-1}}}\right)^{1+\frac{\theta}{2\pi}} \\
& \leq J_{\theta} \left (\sqrt{28}\sin \left (\frac{\pi-\delta}{2} \right )^{-1}\phi(\alpha)^{2} \right )^{1+\frac{\theta}{2\pi}}r^{2^{-j}(1+\frac{\theta}{2\pi})},
\end{align*}
where the last inequality is valid as long as $t_{j}-t_{j-1}\geq t_{j}/2$ (for any
given $k$, this holds for all $1\leq j\leq k$ for sufficiently large
$\alpha$, due to the assumption $r<1$). Now because $r\geq\alpha^{-1}$, we can choose $k$ large
enough (in a way that depends only on $\theta,\varepsilon$) so that one has 
\[
r^{(2^{-1}+...+2^{-k})(1+\theta/\pi)}\leq\alpha^{\frac{1}{2}\varepsilon}r^{1+\theta/20}.
\]
It now follows that 
\begin{align*}
 \mathbb{P}\left (\bigcap_{j=1}^{k}E_{j}'\right )~&=\prod_{j=1}^{k}\mathbb{P}\big(E'_{j}|E_{1}'\cap...\cap E_{j-1}'\big)\\
 & = \prod_{j=1}^{k}\mathbb{E}\bigg(\mathbb{P}\big(E_{j}'|E_{1}'\cap...\cap E_{j-1}',B(t_{j-1})\big)\bigg)\\
 & \leq \prod_{j=1}^{k}J_{\theta} \left (\sqrt{28}\sin \left (\frac{\pi-\theta}{2} \right )^{-1}\phi(\alpha)^{2} \right )^{1+\frac{\theta}{2\pi}}r^{2^{-j}(1+\frac{\theta}{2\pi})}\\
 & \leq_{(*)}\frac{1}{2}\alpha^{\frac{1}{2}\varepsilon}\left(r^{\left(2^{-1}+...+2^{-k}\right)}\right)^{1+\theta/\pi}\leq\frac{1}{2}\alpha^{\varepsilon}r^{1+\theta/20}.
\end{align*}
where ({*}) holds for sufficiently large $\alpha$.

Finally, by a union bound via Equation \eqref{eq:eej} we conclude that for $\alpha$ large enough
we have
\begin{align*}
 \mathbb{P}(E)\leq \tfrac{1}{2} \alpha^{\varepsilon}r^{1+\theta/20}+\mathbb{P} \left (F^{C} \right )\leq \tfrac{1}{2}\alpha^{\varepsilon}r^{1+\theta/20}+\alpha^{-2n-1}\leq\alpha^{\varepsilon}r^{1+\theta/20}.
\end{align*}
\end{proof}
We now have, as a direct consequence, the following slight generalization: 
\begin{lem} \label{lem:mainest}
Let $0 \leq s_{1}<s_{2} \leq 1$. For any $0 \leq \theta < \pi$ and $\varepsilon>0$ there exists $C>0$ such
that the following holds: Suppose that $\alpha>C$ and let $W=\W(w_{0},\pi-\theta)$. Let $a,b \in \mathbb{R}^2$ be points with $|a-w_{0}|=r$
and 
\begin{equation}\label{eq:condosc}
|b-a|<\phi(\alpha)\sqrt{s_{2}-s_{1}}.
\end{equation} 
Let $X(t)$ be a Brownian bridge with boundary conditions $X(s_{1})=a$ and $X(s_{2})=b$.
Then, 
\[
\mathbb{P}\left(X(t)\in W,~~\forall t\in[s_{1},s_{2}]\right)\leq\alpha^{\varepsilon} \max \left (\alpha^{-1}, \frac{r}{\sqrt{s_{2}-s_{1}} }\right )^{1+\theta/20}.
\]
\end{lem}

\begin{proof}
This follows directly from Lemma \ref{lem:technicalwedge} by Brownian scaling. 
\end{proof}
\subsection{Proof of Proposition \ref{prop:bounds}}
The next four lemmas correspond to the bounds of Proposition \ref{prop:bounds}. \\

Let us introduce some notation that will be used in the proofs. Let $(r,s) \in \Delta_n \times \Delta_n$ be given, with the corresponding $(t_0,...,t_{2n+1}) = \mathbf{t}(r,s)$. We fix $0 \leq i \leq 2n$ and denote $s_1 = t_i$ and $s_2 = t_{i+1}$. Let $P$ be the orthogonal projection onto the span of $n_r$ and $n_s$ embedded in $\mathbb{R}^2$, and consider the projected Brownian motion $\tilde B(t) := P B(t)$. In what follows, we will allow ourselves to use the notation $B(t)$ in place of $\tilde B(t)$. Let $d_1= P b_{i}$ and $d_2 = P b_{i+1}$ so that $d_1, d_2 \in \mathbb{R}^2$ and consider the events $E_{i}=\left\{ \tilde B(s_{i})=d_{i}\right\}$, $i=1,2$. 

Let $W$ be the wedge corresponding to $P F_s$ and $P F_r$, so that $W \subset \mathbb{R}^2$ is a wedge of opening angle $\pi-\theta$, $\theta:=\theta(r,s)$, and tip at some point $w_0$, so that $d_1, d_2$ are on $\partial W$. Consider the enlarged wedge $W^{'}$, whose normal directions are $n_r, n_s$, and which contains $W$ in a way that the distance between their respective edges is exactly $\frac{\phi(\alpha)^{2}}{\sqrt{\alpha}}$. With this notation, the event $H_i$ is identical to the event
$$
H_i = A := \{B(t) \in W', ~~ \forall t \in [s_1,s_2] \}.
$$

Finally, let $\mathcal{H}_1, \mathcal{H}_2$ each be one of the halfspaces which define $W'$ corresponding to $d_1$ and $d_2$ respectively, in a way that 
\begin{equation}\label{eq:disttowedge1}
\mathrm{dist}(d_1, \mathcal{H}_1) = \frac{\phi(\alpha)^{2}}{\sqrt{\alpha}}, ~ \mbox{ if } 1 \leq i \leq 2n 
\end{equation}
and
\begin{equation}\label{eq:disttowedge2}
\mathrm{dist}(d_2, \mathcal{H}_2) = \frac{\phi(\alpha)^{2}}{\sqrt{\alpha}}, ~ \mbox{ if } 0 \leq i \leq 2n-1 
\end{equation}
(remark that $\mathcal{H}_1$ and $\mathcal{H}_2$ are not necessarily distinct. Possibly, $\mathcal{H}_1 = \mathcal{H}_2$, in which case $W'$ strictly contains $\mathcal{H}_1\cap \mathcal{H}_2$. Moreover note that $\mathcal{H}_1$ is not defined for the case $i=0$ and $\mathcal{H}_2$ is undefined for $i=2n$). \\

The following lemma is equivalent to the bound \eqref{eq:boundH1}. 
\begin{lem}
With the above notation, assume that both conditions \eqref{eq:disttowedge1} and \eqref{eq:disttowedge2} hold. For every $\varepsilon>0$, there exists $C>0$ such that for any $\alpha>C$:
\[
\mathbb{P}\big(A ~ | E_{1},E_{2}\big)\leq\frac{1}{(s_{2}-s_{1})\alpha}\alpha^{\varepsilon}.
\]
\end{lem}

\begin{proof}
Denote $r:=\frac{s_{1}+s_{2}}{2}$ and define $J_1 := [s_1, r]$ and $J_2 := [r, s_2]$. Consider the events
$$
D_i := \bigl \{B(t) \in \mathcal{H}_i, ~~ \forall t \in J_i \bigr \}, ~~ i=1,2.
$$
Clearly, we have
\[
A \subseteq D_1 \cap D_2.
\]
 We condition on $B(r)$, and use the independence of increments of Brownian motion to obtain an upper bound: 
$$
\mathbb{P}\big(A \cap R_\alpha[s_1,s_2] |E_{1},E_{2}\big) \leq \mathbb{\mathbb{E}}\bigg( \mathbb{P} \big ( D_1 \cap R_\alpha[J1] |E_{1},B(r)\big ) \times \mathbb{P}\big(D_2 \cap R_{\alpha}[J_2]|E_{2},B(r)\big) \bigg | E_1, E_2 \bigg).
$$
Note that under the event $R_\alpha[J1]$, we have $|B(r) - B(s_1)| \leq \sqrt{r-s_1} \phi(\alpha)$, which corresponds to condition \eqref{eq:condosc}. We can therefore apply Lemma \ref{lem:mainest} with respect to the half-space $\mathcal{H}_1$ (which we think of as a degenerate wedge with opening angle $\pi$) and the Brownian bridge on the interval $J_1$, taking $\frac{\varepsilon}{3}$ in place of $\varepsilon$. The condition \eqref{eq:disttowedge1} ensures that we may use $r=\frac{\phi(\alpha)^{2}}{\sqrt{\alpha}}$ as an argument to the lemma. We obtain that almost surely,
$$
\mathbb{P} \big ( D_1 \cap R_\alpha[J1] |E_{1},B(r)\big )  \leq\frac{\alpha^{\frac{\varepsilon}{3}}max(\alpha^{-1}, \phi(\alpha)^{2}\alpha^{-\frac{1}{2}})}{\sqrt{\frac{s_{2}-s_{1}}{2}}}
$$
whenever $\alpha > C$ for some constant $C>0$ depending only on $\varepsilon$. In a similar manner, this time relying on the condition \eqref{eq:disttowedge2}, we get 
$$
\mathbb{P} \big ( D_2 \cap R_\alpha[J2] |E_{2},B(r)\big )  \leq\frac{\alpha^{\frac{\varepsilon}{3}}max(\alpha^{-1}, \phi(\alpha)^{2}\alpha^{-\frac{1}{2}})}{\sqrt{\frac{s_{2}-s_{1}}{2}}}.
$$
By assuming that the constant $C$ is large enough, we have $max(\alpha^{-1}, \phi(\alpha)^{2}\alpha^{-\frac{1}{2}})= \phi(\alpha)^{2}\alpha^{-\frac{1}{2}}$,
so combining the last three displays, we get
\[
\mathbb{P}\big(A\cap R_{\alpha}[s_{1},s_{2}]|E_{1},E_{2}\big)\leq\left(\frac{ \alpha^{\frac{\varepsilon}{3}-\frac{1}{2}}\phi(\alpha)^{2}}{\sqrt{\frac{s_{2}-s_{1}}{2}}}\right)^{2}\leq\frac{1}{(s_{2}-s_{1})\alpha}\alpha^{\varepsilon},
\]
where the final inequality holds for $\alpha > C$ large enough. 
\end{proof}
We now derive an analogous estimate for ``edge'' intervals, giving the bound \eqref{eq:boundH2}.
\begin{lem}
With the above notation, assume that at least one of the two conditions \eqref{eq:disttowedge1}, \eqref{eq:disttowedge2} holds. For every $\varepsilon>0$, there exists $C>0$ such that for any $\alpha>C$:
\[
\mathbb{P}(A\cap R_{\alpha}[s_{1},s_{2}]|E_{1}, E_2)\leq\frac{1}{\sqrt{(t_{i}-t_{i-1})\alpha}}\alpha^{\varepsilon}.
\]
\end{lem}
\begin{proof}
If only the condition \eqref{eq:disttowedge1} holds, then we simply ignore the interval $[(s_{1}+s_{2})/2,s_{2}]$, otherwise proceeding as in the previous lemma. Analogously, if \eqref{eq:disttowedge2} holds, then we ignore the other interval.
\end{proof}

We now move on to the improved bounds \eqref{eq:boundH3} and \eqref{eq:boundH4}, which rely on the additional assumption \eqref{eq:condj}. Thus, in what follows we will also make the assumption 
\begin{equation} \label{eq:condj2}
s_{2}-s_{1}\geq\alpha^{1/(10n)}\max\left(\min(|d_1-w_{0}|,|d_2-w_{0}|)^{2},\frac{1}{\alpha}\right).
\end{equation}
Those lemmas lie in the heart of our argument. This is the only part of the proof where we exploit the fact that the opening angle of the wedge is strictly smaller than $\pi$. The next lemma implies the bound \eqref{eq:boundH3}.
\begin{lem}
With the above notation, assume that the conditions \eqref{eq:disttowedge1}, \eqref{eq:disttowedge2} and \eqref{eq:condj2} hold. For every $\varepsilon>0$, there exists $C>0$ such that, for any $\alpha>C$,
\[
\mathbb{P}\big(A\cap R_{\alpha}[s_{1},s_{2}]|E_{1},E_{2}\big)\leq\frac{\alpha^{\varepsilon}}{(s_{2}-s_{1})\alpha} \alpha^{-\frac{\theta}{400n}}.
\]
\end{lem}

\begin{proof}
Since the distribution of a Brownian bridge is invariant under time reversal (up to the initial conditions), we may assume without loss of generality that $|d_1-w_{0}| \leq |d_2-w_{0}|$. Denote $S^{2}=max \left (\frac{1}{\alpha},|d_1-w_{0}|^{2} \right )$. Set $\ell:=s_{1}+S^{2},\ r:=\frac{s_{1}+s_{2}}{2}$. Remark that
$$
A \subset \left \{ B(t)\in \mathcal{H}_{1},~~\forall t\in[s_{1},\ell]\right\} \cap \left \{ B(t)\in W^{'},~~\forall t\in[\ell,r]\right\} \cap \left\{ B(t)\in \mathcal{H}_{2},~~\forall t\in[r,s_{2}]\right\}.
$$
This time we
condition on $B(\ell)$ and on $B(r)$. By independence of the conditioned Brownian motion on the disjoint intervals, we get
\begin{align} \label{eq:indep}
 \mathbb{P}\big( &A\cap R_{\alpha}[s_{1},s_{2}]|E_{1},E_{2}\big)\leq\\
 & \mathbb{\mathbb{E}}\bigg(\mathbb{P} \left (\left\{ B(t)\in \mathcal{H}_{1},~~\forall t\in[s_{1},\ell]\right\} \cap R_{\alpha}[s_{1},\ell]|E_{1},B(\ell)\right )\times \nonumber \\
 & ~~~~ \mathbb{P} \left (\left\{ B(t)\in W^{'},~~\forall t\in[\ell,r]\right\} \cap R_{\alpha}[\ell,r]|B(\ell),B(r)\right )\times \mathbf{1}_G \times \nonumber \\
 & ~~~~ \mathbb{P}\left (\left\{ B(t)\in \mathcal{H}_{2},~~\forall t\in[r,s_{2}]\right\} \cap R_{\alpha}[r,s_{2}]|E_{2},B(r)\right ) \bigg | E_1, E_2 \bigg), \nonumber \\
\end{align}
where
$$
G := \left \{ |B(\ell) - B(s_1)| \leq \alpha^{-2n-1} + \phi(\alpha) S \right \} \supset R_\alpha[s_1,\ell].
$$
We invoke Lemma \ref{lem:mainest} for each one of the three intervals $[s_1, \ell]$, $[\ell, r]$ and $[r, s_2]$. For the first and last intervals, we use the half-spaces $\mathcal{H}_1$ and $\mathcal{H}_2$ respectively in place of the wedge $W$, invoking the lemma with $\theta = 0$. Using $\varepsilon/4$ in place of $\varepsilon$, we get the bounds
\begin{align}\label{eq:boundI1}
P\bigg(\left\{ B(t)\in \mathcal{H}_{1},~~\forall t\in[s_{1},\ell]\right\} \cap R_{\alpha}[s_{1},\ell] \bigg |E_{1},B(\ell)\bigg) ~& \leq \frac{\alpha^{\frac{\varepsilon}{4}} max(\alpha^{-1}, \alpha^{-\frac{1}{2}}\phi(\alpha)^{2})}{S} \\
& \leq \frac{\alpha^{\frac{\varepsilon}{4}} \alpha^{-\frac{1}{2}}\phi(\alpha)^{2}}{S} \nonumber
\end{align}
and
\begin{equation}\label{eq:boundI2}
P\bigg(\left\{ B(t)\in \mathcal{H}_{2},~~\forall t\in[r, s_2]\right\} \cap R_{\alpha}[r,s_2] \bigg |E_{2},B(r)\bigg)\leq \frac{\alpha^{\frac{\varepsilon}{4}} \alpha^{-\frac{1}{2}}\phi(\alpha)^{2}}{\sqrt{\frac{s_{2}-s_{1}}{2}}}.
\end{equation}
In the above, note that the condition \eqref{eq:condosc} is verified since, under the event $R_\alpha[s_1,\ell]$ we have $|B(\ell)-B(s_1)| < \phi(\alpha) \sqrt{\ell-s_1}$ and an analogous reasoning holds for the second application of the lemma.

We would now like to invoke the lemma for the middle interval, with the wedge being $W'$ (this is the only time we invoke the lemma using a wedge which is not a half-space). First, remark that under the event $G$, we have
\begin{align*}
 |B(\ell)-w_{0}| ~& \leq|B(s_{1}+S^{2})-B(s_{1})|+|B(s_{1})-w_{0}| \\
 & \leq \alpha^{-2n-1}+\phi(\alpha)S+S \\
 & \leq 2\phi(\alpha)S.
\end{align*}
We invoke Lemma \ref{lem:technicalwedge} with the choice $r = 2\phi(\alpha)S$. We end up with the almost-sure bound
\begin{align*}
P\bigg(\left\{ B(t)\in W^{'},~~\forall t\in[\ell,r]\right\} \cap R_{\alpha}[\ell,r]\bigg|B(\ell),B(r)\bigg) \times \mathbf{1}_G~ \\
\leq  \frac{\alpha^{\frac{\varepsilon}{4}}max(\alpha^{-1},2\phi(\alpha)S)^{1+\frac{\delta}{20}}}{\sqrt{\frac{s_{2}-s_{1}}{2}-S^{2}}} \leq\frac{\alpha^{\frac{\varepsilon}{4}} \left (2\phi(\alpha)S\right ) ^{1+\frac{\theta}{20}}}{\sqrt{\frac{s_{2}-s_{1}}{4}}}.
\end{align*}
(in the last inequality we used the fact that for a sufficiently large choice of $C < \alpha$, we have $max(\alpha^{-1},2\phi(\alpha)S)=2\phi(\alpha)S$ and the fact that, by definition, $S^2 \leq \alpha^{-1/(10n)} (s_2-s_1)$). By combining the last display with \eqref{eq:indep}, \eqref{eq:boundI1} and \eqref{eq:boundI2}, we get
\[
\mathbb{P}\big(A\cap R_{\alpha}[s_{1},s_{2}]|E_{1},E_{2}\big)\leq\frac{\alpha^{\frac{\varepsilon}{4}-\frac{1}{2}}\phi(\alpha)^{2}}{S}\times\alpha^{\frac{\varepsilon}{4}}\left(\frac{2\phi(\alpha)S}{\sqrt{\frac{s_{2}-s_{1}}{4}}}\right)^{1+\frac{\theta}{20}}\times\frac{\alpha^{\frac{\varepsilon}{4}-\frac{1}{2}}\phi(\alpha)^{2}}{\sqrt{\frac{s_{2}-s_{1}}{2}}}.
\]
Since for every $\varepsilon > 0$, for sufficiently large $\alpha$, one has $16 \phi(\alpha)^6 \leq \alpha^{\varepsilon / 4}$, we get
$$
\mathbb{P}\big(A\cap R_{\alpha}[s_{1},s_{2}]|E_{1},E_{2}\big)\leq \frac{\alpha^\varepsilon}{\alpha (s_2 - s_1)} \left (\frac{S}{\sqrt{s_{2}-s_{1}}}\right )^{\frac{\theta}{20}}.
$$
Finally, the assumption \eqref{eq:condj2} implies
\[
\frac{S}{\sqrt{s_{2}-s_{1}}}\leq \alpha^{-\tfrac{1}{20n}}.
\]
The combination of the two above displays finishes the proof.
\end{proof}
We again have an analogous lemma for \textquotedbl{}edge\textquotedbl{}
intervals: 
\begin{lem}
With the above notation, assume that one of the two conditions \eqref{eq:disttowedge1} and \eqref{eq:disttowedge2} hold, and assume further that \eqref{eq:condj2} holds. Then, for every $\varepsilon>0$, there exists $C>0$ such that for any $\alpha>C$:
\[
\mathbb{P}\bigg(A\cap R_{\alpha}[s_{1},s_{2}] \bigg | E_{1}, E_2 \bigg)\leq\frac{\alpha^{\varepsilon}}{\sqrt{(s_{2}-s_{1})\alpha}} \alpha^{-\frac{\theta}{400n}}.
\]
\end{lem}
\begin{proof}
We proceed as in the previous lemma. If only the condition \eqref{eq:disttowedge1} holds, then we simply ignore the interval $[r,s_{2}]$. Analogously, if \eqref{eq:disttowedge2} holds, we change the order $s_1 \leftrightarrow s_2$ and proceed as before.
\end{proof}

By combining the lemmas of this subsection, we have now completed the proof of Proposition \ref{prop:bounds}.

\section{Proof of Theorem \ref{thm:main}}
In light of Propositions \ref{prop:reduction} and \ref{prop:tildec}, the proof of the main theorem is reduced to estimating an explicit integral over $\Delta_n \times \Delta_n$. 
Define 
\begin{equation} \label{eq:defZa}
Z_{a}=\Bigl\{(z_{1},\dots,z_{2n});z_{1}\geq a,~1-z_{2n}\geq a\mbox{ and }\forall2\leq j\leq2n,~z_{j}-z_{j-1}\geq a\}.
\end{equation}
We will need the following lemma.
\begin{lem} \label{lem:integral}
For $(r,s) \in \Delta_n \times \Delta_n$, recall that $\mathbf{t}(r,s)=(t_{1} (r,s),...,t_{2n} (r,s))$ are the elements of $r \cup s$ in increasing order. For every dimension $n$ there exists $C>0$ such that for any $0<a<e^{-1}$ we have
\[
\int_{\Delta_{n}\times\Delta_{n}}\frac{\mathbf{1}_{\{\mathbf{t}(r,s) \in Z_{a}\}}}{\sqrt{t_{1} (r,s) (1-t_{2n} (r,s))}}\prod_{i=2}^{2n}\frac{1}{t_{i} (r,s) -t_{i-1} (r,s)} dr ds\leq C|\log(a)|^{C}
\]
and 
\[
\int_{\Delta_{n}\times\Delta_{n}}\mathbf{1}_{\{\mathbf{t}(r,s) \notin Z_{a}\}}dr ds\leq Ca.
\]
\end{lem}
The proof of this lemma is postponed to the appendix. We are ready to finish the proof of our main theorem.
\begin{proof}[Proof of Theorem \ref{thm:main}]
Let $Z_{a}$ be defined as in Equation \eqref{eq:defZa}. Fix $\kappa>0$ and let $\alpha$ be large enough so that the bound given by Proposition \ref{prop:tildec} is valid, in other words,
\[
\mathbb{P}\left(\tilde{\CC}_{\alpha,\kappa}(r,s)\right)\leq\alpha^{-2n-1}+\alpha^{-2n}\alpha^{-\kappa/(16000n)}\frac{1}{\sqrt{t_{1} (r,s) (1-t_{2n} (r,s))}}\prod_{i=2}^{2n}\frac{1}{t_{i}(r,s)-t_{i-1}(r,s)}.
\]
Together with the bounds
given by Lemma \ref{lem:integral}, we have 
\begin{align*}
\alpha^{2n}\int_{\Delta_{n}\times\Delta_{n}} ~& \mathbb{P}\left(\tilde{\CC}_{\alpha,\kappa}(r,s)\right)drds \\
& \leq \alpha^{-1}+\alpha^{-\kappa/(16000n)}\int_{\Delta_{n}\times\Delta_{n}}\frac{\mathbf{1}_{\{\mathbf{t}(r,s) \in Z_{\alpha^{-2n-1}}\}}}{\sqrt{t_{1} (r,s) (1-t_{2n} (r,s))}}\prod_{i=2}^{2n}\frac{1}{t_{i} (r,s) -t_{i-1} (r,s)} dr ds \\
& ~~~ +\alpha^{2n}\int_{\Delta_{n}\times\Delta_{n}}\mathbf{1}_{\{\mathbf{t}(r,s) \notin Z_{\alpha^{-2n-1}}\}}dr ds \\
& \leq\alpha^{-1}+C (2n+1) \alpha^{-\kappa/(16000n)}\log(\alpha){}^{C}+C\alpha^{-1},
\end{align*}
for some constant $C>0$ which depends only on $n,\kappa$. We conclude that 
\[
\alpha^{2n}\int_{\Delta_{n}\times\Delta_{n}}\mathbb{P}\left(\tilde{\CC}_{\alpha,\kappa}(r,s)\right)drds\xrightarrow{\alpha\to\infty}0.
\]
An application of Proposition \ref{prop:reduction} finishes the proof. 
\end{proof}

\section{Appendix: Proofs of technical results}
\begin{proof}[Proof of Lemma \ref{lem:Ralpha}]
	We will show that $\mathbb{P}(Y_{\alpha}^{C})$ and $\mathbb{P}(N_{\alpha}^{C})$
	both satisfy the above seperately, and then apply a union bound. To
	see that this is the case for $\mathbb{P}\big(Y_{\alpha}^{C}\big)$,
	we make use of the following inequality, whose proof can be found
	in \cite{FN}: 
	\[
	E\big(M_\delta^{p}\big)\leq K_{p}\big(\delta ln(1/\delta)\big)^{p/2},
	\]
	for some constant $K_{p}$, valid for every $p>1$ and $\delta<1/e$.
	We take $\delta=\alpha^{-6n-3},p=4$ and conclude that by Markov's
	inequality we have 
	\[
	\mathbb{P}\big(M_{\alpha^{-6n-3}}>\alpha^{-2n-2}\big)=\mathbb{P}\big(M_{\alpha^{-6n-3}}^{4}>\alpha^{-8n-8}\big)\leq
	\]
	\[
	\frac{K_{4}(\alpha^{-6n-3}\log(\alpha^{6n+3}))^{2}}{\alpha^{-8n-8}}=K_{4}\alpha^{-4n+2}\log(\alpha^{6n+3})^{2}.
	\]
	
	So clearly $\mathbb{P}\big(M(\alpha^{-6n-3})>\alpha^{-2n-2}\big)=o(\alpha^{-2n-1})$.
	Therefore $\mathbb{P}\big(M(\delta)\leq\delta^{1/2}\phi(\alpha)+\alpha^{-2n-1},~~\forall\delta<\alpha^{-6n-3}\big)=1-o(\alpha^{-2n-1})$.
	Now let $g_{\alpha}=\lceil\alpha^{6n+4}\rceil$, and consider the
	collection of points $I_{\alpha}=\left\{ \frac{k}{g_{\alpha}}\mid k\in\mathbb{Z},0\leq k\leq g_{\alpha}\right\} $.
	For any pair of points $\frac{k_{1}}{g_{\alpha}},\frac{k_{2}}{g_{\alpha}}\in I_{\alpha}$
	we have: 
	\[
	\mathbb{P}\left ( \left |B(\frac{k_{1}}{g_{\alpha}})-B(\frac{k_{2}}{g_{\alpha}}) \right |>\log(\alpha)\sqrt{\frac{|k_{1}-k_{2}|}{g_{\alpha}}}\right )=\mathbb{P}\big(X>\log(\alpha)\big),
	\]
	
	where $X\sim N(0,1)$. By a standard estimate for Gaussian random
	variables we have 
	\[
	\mathbb{P}\big(X>\log(\alpha)\big)\leq\frac{1}{\sqrt{2\pi}\log(\alpha)}\exp \left (\frac{-\log(\alpha)^{2}}{2}\right ),
	\]
	
	therefore, by a union bound, we have 
	\[
	\mathbb{P}\left (\exists k_{1},\ k_{2}\in I_{\alpha}s.t\ \ \left |B(\frac{k_{1}}{g_{\alpha}})-B(\frac{k_{2}}{g_{\alpha}})\right |>\log(\alpha)\sqrt{\frac{|k_{1}-k_{2}|}{g_{\alpha}}}\right )\leq
	\]
	\[
	g_{\alpha}^{2}\frac{1}{\sqrt{2\pi}\log(\alpha)}\exp \left (\frac{-\log(\alpha)^{2}}{2} \right )=o(\alpha^{-2n-1}).
	\]
	
	For any $x\in\left[0,1\right]$, let $I_{\alpha}(x)$ be a point in
	$I_{\alpha}$ of minimal distance to $x$. Clearly $|x-I_{\alpha}(x)|\leq\alpha^{-6n-4}$.
	Thus we have, with probability at least $1-o(\alpha^{-2n-1})$, that
	for any $s,t\in\left[0,1\right]$ such that $|s-t|>\alpha^{-6n-3}$:
	\[
	|B(s)-B(t)|\leq|B(s)-B(I_{\alpha}(s))|+|B(I_{\alpha}(s))-B(I_{\alpha}(t))|+|B(I_{\alpha}(t))-B(t)|\leq
	\]
	\[
	2\alpha^{-2n-2}+\log(\alpha)\sqrt{|I_{\alpha}(s)-I_{\alpha}(t)|}\leq\alpha^{-2n-1}+2\log(\alpha)\sqrt{|s-t|}\leq\alpha^{-1}+\phi(\alpha)|s-t|,
	\]
	
	where the final two inequalities hold for $\alpha$ sufficiently large.
	We conclude by a union bound that $\mathbb{P}\big(Y\big)=o(\alpha^{-2n-1})$.
	As for $\mathbb{P}\big(N_{\alpha}^{C}\big)$, let $f_{\alpha}=\lceil\frac{\alpha}{\phi(\alpha)}\rceil^{-1}$,
	and consider the points $0,f_{\alpha},2f_{\alpha},...,(f_{\alpha}^{-1}-1)f_{\alpha},1$.
	For each $1\leq j\leq f_{\alpha}^{-1}$, we have that with probility
	$1-e^{-\alpha f_{\alpha}}=1-o(\alpha^{-2n-2})$, the Poisson process
	$\Lambda_{\alpha}$ has at least one point in the interval $\left[(j-1)f_{\alpha},jf_{\alpha}\right]$,
	so by a union bound with probability at least $1-o(\alpha^{-2n-1})$,
	$\Lambda_{\alpha}$ intersects all of these intervals, in which case
	$N_{\alpha}$ clearly holds. 
\end{proof}
\bigskip
\begin{proof}[Proof of Lemma \ref{lem:specialindex}]
	Assume towards condradiction that for all $0\leq j\leq2n-1$, one has that $t_{j+1}-t_{j}\leq\alpha^{1/(10n)}\max\big(|b_j-w_{0}|^{2},\frac{1}{\alpha}\big)$.
	Assume without loss of generality that $t_{j_{0}}\leq1/2$ (otherwise we can make the transformation $t\leftrightarrow1-t$),
	and consider the sequence $a_{1}=t_{j_{0}+1}-t_{j_{0}}$, $a_{2}=t_{j_{0}+2}-t_{j_{0}}$,...,$a_{m}=1-t_{j_{0}} \geq 1/2$
	where $m\leq2n$. For $1\leq i\leq m$ we have the inequalities:
	\[
	a_{i+1}-a_{i}=t_{i+1}-t_{i}\leq\alpha^{1/10n}\max\left (|b_i-w_{0}|^{2},\frac{1}{\alpha}\right ),
	\]
	and
	\begin{align*}
	|b_i-w_{0}|^{2}\leq\big(|b_i-b_{j_{0}}|+|b_{j_{0}}-w_{0}|\big)^{2} ~& \leq_{(1)} \left (\phi(\alpha)(t_{i}-t_{j_{0}})^{1/2}+\alpha^{-2n-1}+\text{M}\frac{\phi(\alpha)^{2}}{\sqrt{\alpha}}\right )^{2} \\
	& \leq_{(2)}16M^{2}\phi(\alpha)^{4}\max((t_{i}-t_{j_{0}}),(1/\alpha)),
	\end{align*}
	where the inequalities (1) and (2) follow from the property \eqref{eq:prop1} and
	the elemantary inequality $(a^{1/2}+b^{1/2})^{2}\leq4max(a,b)$ respectively.
	Thus we have, for all $0\leq i\leq m$,
	\[
	a_{i+1}-a_{i}\leq16M^{2}\alpha^{1/10n}\phi(\alpha)^{4}\max(a_{i},1/\alpha).
	\]
	
	Futhermore by assumption we have: 
	\begin{align*}
	a_{1}=t_{j_{0}+1}-t_{j_{0}} ~& \leq\alpha^{1/10n}\max\left (|b_{j_{0}}-w_{0}|^{2},\frac{1}{\alpha}\right ) \\
	& \leq \alpha^{1/10n}M^{2}\frac{\phi(\alpha)^{4}}{\alpha}.
	\end{align*}
	
	We claim that for $\alpha$ larger than some constant, this implies
	that $a_{i}<1/2$ for all $i$, which is a contradiction. Indeed,
	if $a_{i}\leq1/\alpha$ for all $i$ then we're done, assuming $\alpha>2$.
	Otherwise let $1\leq k\leq m$ be the minimal index such that $a_{k}>1/\alpha$.
	We then have 
	\[
	a_{k}\leq a_{1}+\frac{(k-1)16M^{2}\alpha^{1/10n}\phi(\alpha)^{4}}{\alpha}\leq32nM^{2}\alpha^{1/10n}\frac{\phi(\alpha)^{4}}{\alpha},
	\]
	and by induction
	\[
	a_{m}\leq a_{k}\big(16M^{2}\alpha^{1/10n}\phi(\alpha)^{4}+1\big)^{m-k}\leq32nM^{2}\alpha^{1/10n}\frac{\phi(\alpha)^{4}}{\alpha}\big(16M^{2}\alpha^{1/10n}\phi(\alpha)^{4}+1\big)^{2n}
	\]
	\[
	<\alpha^{-1+1/5n+2/5}<1/2,
	\]
	where the above inequalities hold for $\alpha$ larger than some constant. 
\end{proof}
\bigskip
\begin{proof}[Proof of Lemma \ref{lem:integral}]
	In what follows, we will allow ourselves to abbreviate $t_1 = t_1(r,s), t_2 = t_2(r,s)$ etc. To prove the first bound, we first note that the integrand is independent
	of the ordering of the elements in $r \cup s$. Furthermore, $\Delta_{n}\times\Delta_{n}$ can be written as the union
	of the following $\binom{2n}{n}$ sets, which are disjoint up to sets
	of measure zero: 
	\[
	\Delta_{n}\times\Delta_{n}=\bigcup_{1\leq k_{1}<...<k_{n}\leq2n}\left\{ (x_{1},..,x_{2n})\in\Delta_{n}\times\Delta_{n};x_{1}=t_{k_{1}},...,x_{n}=t_{k_{n}}\right\} .
	\]
	
	Noting that taking $k_{1}=1,...,k_{n}=n$ results in the domain $\Delta_{2n}\subseteq\Delta_{n}\times\Delta_{n}$,
	wherein we have $x_{i}=t_{i}$ for all $1\leq i\leq2n$, we get that
	\begin{align*}
	& \int_{\Delta_{n}\times\Delta_{n}}\frac{\mathbf{1}_{\{(t_{1},...,t_{2n})\in Z_{a}\}}}{\sqrt{t_{1}(1-t_{2n})}}\prod_{i=2}^{2n}\frac{1}{t_{i}-t_{i-1}}drds=\\
	& \binom{2n}{n}\int_{\Delta_{2n}}\frac{\mathbf{1}_{\{(x_{1},...,x_{2n})\in Z_{a}\}}}{\sqrt{x_{1}(1-x_{2n})}}\prod_{i=2}^{2n}\frac{1}{x_{i}-x_{i-1}}dx_{1}...dx_{2n}\leq\\
	& \binom{2n}{n}\int_{[0,1]^{2n}}\frac{\mathbf{1}_{\{(x_{1},...,x_{2n})\in Z_{a}\}}}{\sqrt{x_{1}(1-x_{2n})}}\prod_{i=2}^{2n}\frac{1}{|x_{i}-x_{i-1}|}dx_{1}...dx_{2n}\leq_{(*)}\\
	& \frac{1}{2}\binom{2n}{n}\int_{\left[0,1\right]^{2n}}\frac{\mathbf{1}_{\{(x_{1},...,x_{2n})\in Z_{a}\}}}{x_{1}}\prod_{i=2}^{2n}\frac{1}{|x_{i}-x_{i-1}|}dx_{1}\cdot\cdot\cdot dx_{2n}+\\
	& \frac{1}{2}\binom{2n}{n}\int_{\left[0,1\right]^{2n}}\frac{\mathbf{1}_{\{(x_{1},...,x_{2n})\in Z_{a}\}}}{1-x_{2n}}\prod_{i=2}^{2n}\frac{1}{|x_{i}-x_{i-1}|}dx_{1}\cdot\cdot\cdot dx_{2n},
	\end{align*}
	
	where ({*}) follows from the arithmetic-geometric mean inequality. We will prove the required
	bound for the first summand, with the second one being completely
	analogous. We make the change of variables $y_{1}=x_{1}$, $\forall2\leq j\leq2n,~y_{j}=x_{j}-x_{j-1}$:
	\begin{align*}
	& \int_{\left[0,1\right]^{2n}}\frac{\mathbf{1}_{\{(x_{1},...,x_{2n})\in Z_{a}^{'}\}}}{x_{1}}\prod_{i=2}^{2n}\frac{1}{|x_{i}-x_{i-1}|}dx_{1}\cdot\cdot\cdot dx_{2n}=\\
	& \int_{Q}\mathbf{1}_{\{(y_{1},...,y_{2n});\forall1\leq j\leq2n,\ y_{j}\geq a\}}\prod_{i=1}^{2n}\frac{1}{y_{i}}dy_{1}\cdot\cdot\cdot dy_{2n},
	\end{align*}
	
	where $Q:=\{(y_{1},...,y_{2n});\forall1\leq j\leq2n,\ \ 0\leq\sum_{1\leq i\leq j}y_{i}\leq1\}$.
	We can again upper bound this by replacing $Q$ with $\left[0,1\right]^{2n}$,
	and finally by Fubini's Theorem we have 
	\[
	\int_{\left[0,1\right]^{2n}}\mathbf{1}_{\{(y_{1},...,y_{2n});\forall1\leq j\leq2n,~y_{j}\geq a\}}\prod_{i=1}^{2n}\frac{1}{y_{i}}dy_{1}\cdot\cdot\cdot dy_{2n}=|\log(a)|^{2n},
	\]
	as required. To prove the second inequality we note by the same reasoning
	that 
	\[
	\int_{\Delta_{n}\times\Delta_{n}}\mathbf{1}_{\{(t_{1},...,t_{2n})\notin Z_{a}\}}dr ds=\binom{2n}{n}\int_{\Delta_{2n}}\mathbf{1}_{\{(x_{1},...,x_{2n})\notin Z_{a}\}}dx_{1}\cdot\cdot\cdot dx_{2n}.
	\]
	
	Now 
	\[
	\mathbf{1}_{\{(x_{1},...,x_{2n})\notin Z_{a}\}}\leq\mathbf{1}_{\left\{ x_{1}\leq a\right\} }+\mathbf{1}_{\left\{ 1-x_{2n}\leq a\right\} }+\sum_{j=2}^{2n}\mathbf{1}_{\left\{ x_{j}-x_{j-1}\leq a\right\} },
	\]
	
	and we can again focus on proving the bound for one summand, say $\mathbf{1}_{\left\{ x_{1}\leq a\right\} }$,
	with all the others being analogous. To do this we simply apply Fubini's
	Theorem, obtaining that 
	\[
	\int_{\Delta_{n}\times\Delta_{n}}\mathbf{1}_{\left\{ x_{1}\leq a\right\} }dx_{1}\cdot\cdot\cdot dx_{2n}\leq\int_{\left[0,1\right]^{2n}}\mathbf{1}_{\left\{ x_{1}\leq a\right\} }dx_{1}\cdot\cdot\cdot dx_{2n}=a.
	\]
\end{proof}

\end{document}